\pdfoutput=1

\documentclass[12pt]{article}
\usepackage{graphicx}
\usepackage{setspace}
\usepackage{amssymb}
\usepackage{amsmath}
\usepackage{bbm} 
\usepackage{bm} 
\usepackage{subfigure}
\usepackage{cite} 
\usepackage{hyperref}

\newtheorem{theorem}{Theorem}
\newtheorem{proposition}[theorem]{Proposition}
\newtheorem{lemma}[theorem]{Lemma}
\newtheorem{corollary}[theorem]{Corollary}
\newenvironment{proof}[1][Proof]{\textbf{#1.} }{\ \rule{0.5em}{0.5em}}

\def\I{{\rm i}}

\def\D{{\rm d}}

\begin{document}

\title{Spatio-spectral limiting on hypercubes: eigenspaces}
\author{Jeffrey A. Hogan\\
School of Mathematical and Physical Sciences\\
University of Newcastle\\
Callaghan NSW 2308 Australia\\
email: {\tt jeff.hogan@newcastle.edu.au} \and
Joseph D. Lakey\\
Department of Mathematical Sciences\\
New Mexico State University\\
Las Cruces, NM 88003--8001\\
email: {\tt jlakey@nmsu.edu}}
\abstract{
The operator that first truncates to a neighborhood of the origin in the spectral domain then truncates
to a neighborhood of the origin in the spatial domain is investigated in the case of Boolean cubes. 
This operator is self adjoint on a space of bandlimited signals.   The eigenspaces of this iterated projection operator are 
studied and are shown to depend fundamentally on the neighborhood structure of the cube when regarded as a 
metric graph with path distance equal to Hamming distance.
\smallskip\noindent
{\bf keywords:\,}{Boolean cube, Slepian sequence, 
time and band limiting, graph Laplacian, graph Fourier transform
}

} 
\maketitle

The main result of this work identifies eigenspaces of joint spatio--spectral limiting on Boolean hypercubes.
Corresponding distributions of eigenvalues will be described in other work.
Spatial limiting refers, broadly, to truncation of a function to a neighborhood of a point. Spectral limiting refers to restriction
to certain eigenmodes.  The theory of joint spatio--spectral limiting on $\mathbb{R}$---the theory of time and band limiting---was developed in a series of works by Landau, Slepian and Pollak  \cite{slepian_pollak_I,landau_pollak_II,landau_pollak_III,slepian_IV,slepian_V_1978} that appeared in the \emph{Bell Systems Tech. Journal} in the 1960s. This theory first identified the eigenfunctions of time- and band-limiting operators, compositions $P_\Omega Q_T$ where $(Q_T f)(t)=\mathbbm{1}_{[-T,T]}(t) \, f(t)$ and $(P_\Omega f)(t)=(Q_{\Omega/2}\widehat{f})^{\vee}(t)$ (here $\widehat{f}(\xi)=\int_{-\infty}^\infty f(u)\, e^{-2\pi i u\xi}\, d\xi$), \cite{slepian_pollak_I}, and eventually quantified the distribution of eigenvalues of $P_\Omega Q_T$ \cite{landau_widom}. 
 Extensions to $\mathbb{R}^n$ and the compact--discrete setting were developed by Slepian \cite{slepian_IV,slepian_V_1978}
 and finite dimensional versions on $\mathbb{Z}_N$ by Gr\"unbaum  \cite{grunbaum_toeplitz,grunbaum_linalg_1981} and by Xu and Chamzas  \cite{xu_chamzas_dpss}, cf., \cite{jain_ranganath_1981}.  Since 2000, besides refined study of 
 eigenfunctions of time and band limiting, e.g.,\cite{xiao_rokhlin_yarvin_pswf,boyd_pswf_2003_2,boyd_2005,osipov2013prolate,Osipov2014108,sc12-1,Bonami2014229,Bonami2016,zhu_etal_2017},
 a variety of applications of time- and band-limiting methods to  communication channel modeling, e.g., \cite{zemen_mecklenbraucker_2005,farrell_strohmer_2011},  multiband signals \cite{davenport_wakin_2012},
 irregular sampling \cite{bass2013},  and super-resolution limits \cite{CPA:CPA21455}, have been developed.
    These developments all make important use of Euclidean harmonic analysis (or analogous methods on the circle, the integers, and the integers modulo $N$).

 More recently, aspects of the theory developed in classical settings have been developed in more general settings such as the 2-sphere by Simons et al.~\cite{wieczorek_simons,simons_2012},  locally compact abelian groups by Zhu and Wakin \cite{zhu_wakin_abelian} and, at least in an experimental sense, finite graphs $G=(V,E)$, see Tsitsvero et al.~\cite{tsitsvero_etal_2015,tsitsvero_etal_2015_2}.  It is  unreasonable to hope for an explicit description of eigenvectors of joint spatiospectral limiting on a graph $G$
unless $G$ has a particular structure with ample symmetry, such as a when $G$ is the Cayley graph of a finite abelian group, see \cite{Hogan20172}.  Our main result, Cor.~\ref{cor:qpq_evec} identifies the eigenspaces of spatio-spectral limiting operators on hypercubes,
which are identified with Cayley graphs of $\mathbb{Z}_2^N$.

\section{The Boolean Hypercube, space limiting and band limiting \label{sect:boolean}}

The Boolean hypercube $\mathcal{B}_N$ is the set $\{0,1\}^N$.  It can be regarded as the Cayley graph of the  group of $(\mathbb{Z}_2)^N$ with symmetric generators $e_i$ having entry 1 in the $i$th coordinate and zeros in the other $N-1$ coordinates.  With componentwise addition modulo one, vertices corresponding to elements of $(\mathbb{Z}_2)^N$ share an edge precisely when their difference in $(\mathbb{Z}_2)^N$ is equal to $e_i$ for some $i$.
It will be convenient to index vertices  of $\mathcal{B}_N$  by subsets $S$  of $\{1,\dots, N\}$ according to coordinates having bit-value one. That is, if $\mathbf\epsilon=(\epsilon_1,\dots,\epsilon_N)\in \{0,1\}^N$ has $r$ nonzero \emph{bit indices} $\beta_1,\dots, \beta_r$ in $\{1,\dots, N\}$, we identify  $S=\{\beta_1,\dots,\beta_r\}$ as those indices such that $\epsilon_{\beta_i}=1$.
Distance between vertices is defined by Hamming distance---the number of differing coordinates---which is equal to the path distance when each edge has unit weight.   The unnormalized graph Laplacian of $\mathcal{B}_N$ is the matrix $L$, thought of as a function on 
$\mathcal{B}_N\times \mathcal{B}_N$, with $L_{SS}=N$ and $L_{RS}=-1$ if $R\sim S$, that is, $R$ and $S$ are nearest neighbors---they differ in a single coordinate---and $L_{RS}=0$ otherwise.   Thus $L=NI-A$ where $A$ is the adjacency matrix $A_{RS}=1$ if $R\sim S$ and $A_{RS}=0$ otherwise. It will be convenient to carry out our analysis on the unnormalized Laplacian, though the normalized Laplacian   $\mathcal{B}_N$ is $\mathcal{L}=L/N$ is preferred for  analysis asymptotic in $N$. We denote by $|S|$ the number of elements of
$S$, equivalently, the number of ones in the element of $\mathcal{B}_N$ determined by $S$.   The 
eigenvalues of $L$ are $0,2,\dots, 2N$ where $2K$ has multiplicity $\binom{N}{K}$.  The corresponding eigenvectors,
which together form the graph (and group) Fourier transform, are the 
Hadamard vectors defined by $H_S(R)=(-1)^{|R\cap S|}$ . The normalized Hadamard vectors
$\bar{H}_S=H_S/2^{N/2}$  together make the Fourier transform unitary. 
Lemma \ref{lem:hadamard_eigenvalue}, which states that $H_S$ is an eigenvector of $L$ with eigenvalue $2|S|$, is well known both in group theory and graph theory.
A combinatorial proof  compares intersection parities $|R\cap S|{\rm mod}\, 2$ and  $|P\cap S|{\rm mod}\, 2$
among  $P\sim R$, e.g., \cite{Hogan20172}.

\begin{lemma}\label{lem:hadamard_eigenvalue} 
$H_S$ is an eigenvector of $L$ with eigenvalue
$2|S|$.
\end{lemma}

We denote by $H$ the matrix with columns $H_S$, $S\subset\{1,\dots, N\}$. Up to an indexing of the columns and a factor $2^{N/2}$, $H$ is the \emph{Hadamard matrix} of order $N$ obtained by taking the $N$-fold tensor product of the matrix 
\begin{tiny}$\left(\begin{array}{cc} 1 & 1 \\ 1& -1\end{array}\right)$.\end{tiny}
On $\mathbb{R}$, the timelimiting operator is $(Q_T f)(t)=\mathbbm{1}_{[-T,T]}(t) f(t)$ which cuts off $f$ in a neighborhood
of the origin.  In $\mathcal{B}_N$, cutting off in a (symmetric) neighborhood of the origin---the vertex corresponding to $(0,\dots,0)\in\mathbb{Z}_2^N$---means multiplying by the 
characteristic function of a Hamming ball of a given radius. The closed Hamming ball $B(0,r)$ of radius $r$ consists
of those $R\in \mathcal{B}_N$ such that $|R|\leq r$.  The Hamming sphere $\Sigma_r$ consists of those vertices having Hamming distance $r$ from the origin. They are indexed by the $r$-element subsets of $\{1,\dots, N\}$ and thus have $\binom{N}{r}$ vertices.  Consequently, $B(0,K)$ has ${\rm dim}(K)=\sum_{k=0}^K \binom{N}{k}$ vertices.

On $\mathcal{B}_N$, one denotes $Q=Q_K$ by 
$(Qf)(S)=\mathbbm{1}_{B(0,K)} f(S)$.
On $\mathbb{R}$, it is typical to denote the $\Omega$-bandlimiting
operator $(P_\Omega f)(t)= (\widehat{f}(\xi)\mathbbm{1}_{[-\Omega/2,\Omega/2]}(\xi))^\vee$ where $\widehat{f}$ denotes the Fourier transform of $f$. By analogy one can define a bandlimiting operator $P=P_K$ by $P=\bar{H}Q\bar{H}$.  Equivalently, $P_K$ is the projection onto the span of those $H_S$ with $|S|\leq K$.  By analogy with Paley--Wiener spaces on $\mathbb{R}$, we refer to the span of $\{H_S: |S|\leq K\}$ as the Boolean \emph{Paley--Wiener} space ${\rm BPW}_K$, which has dimension ${\rm dim}(K)$.

In this work we seek to characterize the eigenspaces of the joint projection operator $P_{K_1} Q_{K_2}P_{K_1}$
which is the $\mathcal{B}_N$-analogue of the \emph{time- and band-limiting operator} $P_\Omega Q_T P_\Omega$
on $\mathbb{R}$.  To simplify the presentation somewhat we will restrict to the case of fixed $K_1=K_2=K$. The eigenvalues of $P_\Omega Q_T P_\Omega$ on $\mathbb{R}$ are non-degenerate, as are the eigenvalues of their finite dimensional ($\mathbb{Z}_N$) analogues \cite{slepian_pollak_I,grunbaum_toeplitz}.  These facts are tied to the simple
linear geometry of $\mathbb{R}$. In contrast, the eigenvalues of $P_{K} Q_{K}P_{K}$ on $\mathcal{B}_N$
have high multiplicity, namely 
$\binom{N}{k}-\binom{N}{k-1}$, $0\leq k\leq K$, 
and there are $(K+1-k)$ distinct eigenvalues with such multiplicity.  
Summation by parts gives
\[ \sum_{k=0}^K (K+1-k) \left(\binom{N}{k}-\binom{N}{k-1}\right) =\sum_{k=0}^K \binom{N}{k} ={\rm dim}\, (K) \, .
\]
Thus the eigenspaces of $PQP$ provide a decomposition of ${\rm BPW}_K$.  This decomposition will be done in the frequency domain by identifying
certain spaces of vectors that are invariant under the adjacency matrix $A$, and expressing the bandlimiting operator as a polynomial in $A$.
Any vector in this space is described in terms of at most $N+1$ coefficients and the eigenvectors of $PQP$ are determined by certain eigencoefficients.
As such, the eigenspace problem is reduced from a problem of finding eigenvectors of a matrix of size on the order of $2^N$ to that of finding eigenvectors
of a matrix of size on the order of $N$.

Before addressing this problem for $PQP$, we will solve the corresponding problem for a certain second-order difference operator analogue of the so-called
\emph{prolate differential operator} on $\mathbb{R}$,
\begin{equation}\label{eq:pdo} {\rm PDO}=\frac{\D}{\D t}(t^2-T^2) \frac{\D}{\D t}+ (\pi\Omega)^2 t^2\, 
\end{equation}
that commutes with the operator that truncates $[-1,1]$ and bandlimits to $[-\Omega/2,\Omega/2]$. 
Up to dilation, the eigenfunctions of PDO (and hence of time and bandlimiting) are the so-called  prolate spheroidal wave functions (PSWFs) \cite{slepian_pollak_I,hogan_lakey_tbl}.  
The structure of PDO allows for efficient numerical computation of the PSWFs.  
It is reasonable to seek a parallel route  to identify eigenvectors of $PQP$ in the Boolean case.

The identity $(\frac{\D}{\D t} f)^{\wedge}(\xi)= 2\pi \I \xi\hat{f}(\xi)$ means that, up to a constant,
differentiation $\sqrt{-\Delta}$, is conjugation of multiplication by $\xi$, the corresponding eigenvalue, with the Fourier transform.
The operator $D=\bar{H}T\bar{H}$ that conjugates the diagonal matrix $T$ whose diagonal elements are square roots of eigenvalues of $L$,
by $H$ can thus be regarded, up to normalization, as a Boolean analogue of $d/dt$ while 
$T$ can also be regarded as an analogue, again up to normalization, of multiplication by $t$. 
The \emph{Boolean difference operator}
\begin{equation}\label{eq:bdo} ({\rm BDO})\qquad D(\alpha I-T^2) D+\beta T^2\, 
\end{equation}
then can be regarded as a Boolean analogue of PDO in (\ref{eq:pdo}) where $\alpha$ and $\beta$ play normalizing roles.

We will show that the eigenspaces of BDO are identified at the coefficient level with spaces of eigenvectors of a certain tri-diagonal matrix.  Unlike the case of the real line, BDO and $PQP$ do not commute \cite{Hogan20172}.
 Nonetheless, BPDO and $PQP$ share invariant subspaces that enables similar analysis for each, but 
the eigenspaces of $PQP$ are more complicated because the coefficient matrix in question is not tri-diagonal.

\subsection{Dyadic lexicographic order}

In order to represent linear operators on $\ell^2(\mathcal{B}_N)$ one needs a way of listing the elements of 
$\mathcal{B}_N$ in a particular order.  One way is to start with $S\sim (\epsilon_0,\dots, \epsilon_{N-1})$ and associate
$n(S)=\sum_{j=0}^{N-1} \epsilon_j 2^j$.  However, to describe operations on  $\mathcal{B}_N$ relative to Hamming distance it is preferable 
to work with  an order that respects  Hamming distance to the origin.  The dyadic lexicographic order is defined on elements of $
\mathcal{B}_N$ as follows. Recall that $R$ is identified with a subset of $\{1,\dots, N\}$ so that the elements of $R$ are those $\beta$ such 
that $\epsilon_{\beta}=1$ when $R\sim (\epsilon_0,\dots, \epsilon_{N-1})$. The dyadic lexicographic order ``$\lesssim$'' stipulates $R\lesssim S$ if $|R|<|S|$ and,
when $|R|=|S|$ if, in the smallest bit index $\beta$ in which $R$ and $S$ differ, one has $\beta\in S\setminus R$.
In this ordering, the adjacency matrix $A$ can be represented as a block symmetric matrix with nonzero entries
restricted to the blocks corresponding to products of Hamming spheres $\Sigma_r\times\Sigma_{r\pm 1}$.  

\begin{figure}[ht]
\centering 
\includegraphics[width=3.5in,height=2.2in]{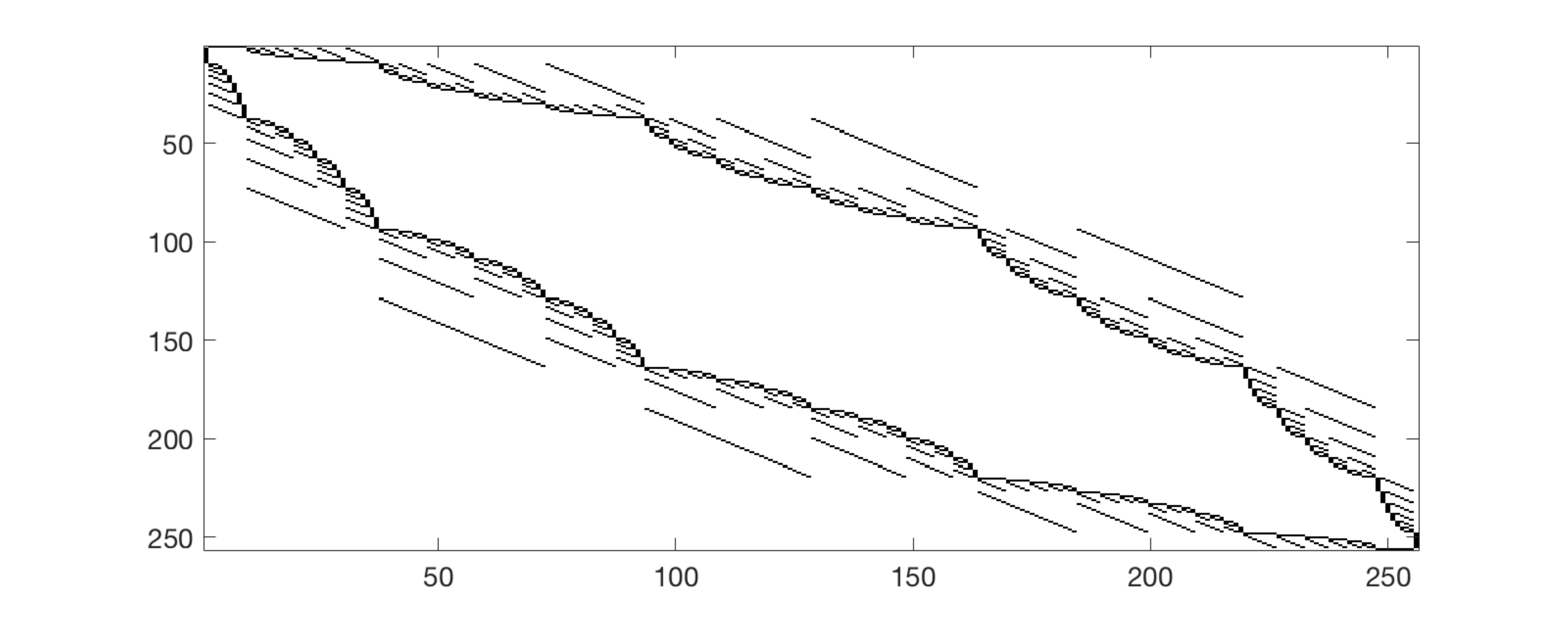}
\caption{Adjacency matrix for $N=8$ in dyadic lexicographic order.}
\label{fig:norm_pqp_bde_7_3}
\end{figure}

\section{Hamming sphere analysis of adjacency}

\subsection{Compositions of outer and inner adjacency maps}
The adjacency matrix $A$ of $\mathcal{B}_N$ can be written $A=A_{+}+A_{-}$ where $A_-=A_+^T$ and
$A_+$ is lower triangular when expressed in dyadic lexicographic order. 
Thus $A_+$ maps data on $r$-spheres to data on $(r+1)$-spheres while $A_-$ maps from $r$-spheres to $(r-1)$-spheres.
We may refer to $A_+$ and  $A_-$ as the respective \emph{outer} and \emph{inner adjacency maps}. 

The space $\ell^2(\Sigma_r)$ of vectors supported on $\Sigma_r$ has an orthogonal decomposition 
\[\ell^2(\Sigma_r)=A_+ \ell^2(\Sigma_{r-1})\oplus \mathcal{W}_r
\]
where $\mathcal{W}_r$ is the orthogonal complement of $A_+ \ell^2(\Sigma_{r-1})$ inside $\ell^2(\Sigma_r)$.
Applying such a decomposition at each scale, one obtains a type of multiscale decomposition:
\[\ell^2(\Sigma_r)=A_+ \ell^2(\Sigma_{r-1})\oplus \mathcal{W}_r=\cdots= A_+^r\mathcal{W}_0\oplus A_+^{r-1}\mathcal{W}_1\oplus\cdots\oplus \mathcal{W}_r\, .
\]
Since the origin is a singleton, $\mathcal{W}_0$ is simply the constants at the origin.

The outer and inner adjacencies do not commute.
The following theorem quantifies composition of $A_-$ with a power of $A_+$.

\begin{theorem}\label{thm:wradjacency}
Let $W\in\mathcal{W}_r$ and $k$ such that $r+k<N$.  Then
\[A_-A_+^{k+1} W=[(N-2r)+(N-2(r+1))+\cdots +(N-2(r+k))]\,  A_+^k W
\, .
\]
\end{theorem}

We denote by $m(r,k)$ the multiplier of $A_+^k W$, that is,
\begin{equation}\label{eq:mrk}m(r,k)=[(N-2r)+(N-2(r+1))+\cdots +(N-2(r+k))] \, . \end{equation}
The proof uses the following lemma whose proof will be given after that of Thm.~\ref{thm:wradjacency}.

\begin{lemma} \label{lem:commutator} Let $C=[A_-,A_+]=A_-A_+-A_+A_- $ be the commutator of $A_-$ and $A_+$. Then for each $r$,
the restriction of $C$ to $\Sigma_r$ is multiplication by $N-2r$.
\end{lemma}

\begin{proof}[Proof of Thm.~\ref{thm:wradjacency}]
We first prove the cases $k=0$ and $k=1$. The general case is then proved by induction on $k$ for
each fixed $r$. The case $k=1$ thus does not require separate proof, but we include that case to illustrate
how the orthogonality condition defining $\mathcal{W}_r$ propagates through powers of $A_+$.
For $k=0$ the claim is that $A_-A_+W=(N-2r) W$. If $R\in\Sigma_r$ then the value 
$(A_-A_+W)(R)=\sum_{S\sim\sim R} W_S$, counting multiplicity, where $S\sim\sim R$ means that there
is a path of length two from $S$ to $R$ through a vertex in $\Sigma_{r+1}$, and the multiplicity of the value
$W_S$ is the number of such distinct paths. Equivalently, it is the number of common neighbors that $R$ and $S$
share in $\Sigma_{r+1}$.  Since $R\in\Sigma_r$, it has $N-r$ neighbors in $\Sigma_{r+1}$, one for each bit index
not in $R$.   
Denote by $\beta\in R$ a \emph{bit index} element
of $\{1,\dots, N\}$ and let $(R\cup\{\gamma\}\setminus\{\beta\})$
denote the element of $\Sigma_r$ obtained by replacing $\beta$ by $\gamma$.
If $R=\{\beta_1,\dots, \beta_r\}$ then
\begin{equation}\label{eq:aminusaplus}(A_-A_+W)(R)=(N-r)W +\sum_{i=1}^r \sum_{\beta\notin R} 
W(R\cup\{\beta\}\setminus\{\beta_i\}) 
\end{equation}
where the first term counts those paths of the form $R\mapsto R\cup \{\gamma\}\mapsto R$ where $\gamma$ is one of the 
$N-r$ elements of $\{1,\dots,N\}\setminus R$, and the second term counts those paths through $\Sigma_{r+1}$ that originate
from vertices in $\Sigma_r$ of the form $R\cup\{\beta\}\setminus\{\beta_i\}$ 
where $\beta\notin R$, and terminate at $R$.
The sum on the right in (\ref{eq:aminusaplus}) can also be written
\[\sum_{i=1}^r [W(R)+ \sum_{\beta\notin R} W(R\cup\{\beta\}\setminus\{\beta_i\})]
-rW(R)
=-rW(R)
\]
where, in the last equation we used the fact that $W(R)+ \sum_{\beta\notin R} W(R\cup\{\beta\}\setminus\{\beta_i\})
=0$ because
it is  the inner product of $W$ with $A_+ \delta_{R\setminus\{ \beta_i\}}$. Here, $\delta_S$ is the vertex function taking value one at $S$ and 
zero elsewhere, so   $A_+ \delta_{R\setminus\{ \beta_i\}}$ takes value one at each $\Sigma_r$-neighbor of $R\setminus \{\beta_i\}$, that is, at 
$R$ and each $R\cup\{\beta\}\setminus\{\beta_i\}$, 
$\beta\notin R$, and zero elsewhere.   
 We conclude that 
\[(A_-A_+W)(R)=(N-2r)W(R) \]
for each $R\in\Sigma_r$ as claimed, when $W\in\mathcal{W}_r$.  This proves the case $k=0$.

\bigskip
Next we consider the case $k=1$.
Again let $W\in\mathcal{W}_r$.  The vector $A_-A_+^2$ is supported on $\Sigma_{r+1}$. We will use $S\in\Sigma_{r+1}$
to denote a generic target at which $A_-A_+^2W$ will be evaluated.
The values $(A_-A_+^2W)(S)$ are the values of $W$ that originate at some $R\in\Sigma_r$ and terminate
at $S$ through a path of length three that passes through $\Sigma_{r+2}$.
The values can be organized as 
\begin{multline}\label{mult:aminusaplus2}
(A_-A_+^2W)(S) = \sum_{b\notin S} (A_+^2W)(S\cup\{b\})
=2\sum_{b\notin S} \sum_{\{b_1,b_2\}\subset S\cup\{b\}} W((S\cup\{b\})\setminus \{b_1,b_2\})    \\
=2\sum_{b\notin S} \sum_{\beta\in S }W((S\setminus\{ \beta\})    
+2\sum_{b\notin S} \sum_{\{b_1,b_2\}\subset S\cup\{b\}, b_i\neq b} W((S\cup\{b\})\setminus \{b_1,b_2\})  
=I+II\end{multline}
where the sum in $I$ accounts for the cases in which one of the $b_i$ is equal to the bit index $b$ in the outer sum, 
and in $II$, neither of the $b_i$ is equal to $b$.
 The factor two accounts for the fact that each endpoint of $A_+^2$ is attained through two paths.  In general, a
 vertex value $W(R)$, $R\in\Sigma_r$ contributes to $(A_+^k W)(S)$, $S\in\Sigma_{r+k}$ with multiplicity $k!$ which 
 is the number of paths over which the value at $R$ is transmitted to $S$, one for each permutation of the $k$ elements
 in $S\setminus R$.

 Each inner sum in term I gives the value $(A_+W)(S)$ and there are $N-r-1$ values, one for each $b\notin S$.
   Multiplying by 2 gives $2(N-r-1) (A_+W)(S)$ for the term $I$.
By adding and subtracting terms of the form  $W(S\setminus\{\beta\})$, the sum in the term $II$ is 
\begin{multline*}
\sum_{b\notin S} \sum_{\{b_1,b_2\}\subset S\cup\{b\}, b_i\neq b} W((S\cup\{b\})\setminus \{b_1,b_2\})  
=  \sum_{\{b_1,b_2\}\subset S} \sum_{b\notin S} W((S\cup\{b\})\setminus \{b_1,b_2\})  \\
= \sum_{\{b_1,b_2\}\subset S}\Bigl(\Bigl[ W(S\setminus\{b_1\})
+W(S\setminus\{b_2\})+ \sum_{b\notin S} W(S\cup\{b\}\setminus \{b_1,b_2\})\Bigr]\\
 -W(S\setminus\{b_1\})-W(S\setminus\{b_2\})\Bigr)
= -\sum_{\{b_1,b_2\}\subset S} W(S\setminus\{b_1\})+W(S\setminus\{b_2\})
\end{multline*}
where we have used the fact that 
\[W(S\setminus\{b_1\})+W(S\setminus\{b_2\})+ \sum_{b\notin S} W((S\cup\{b\})\setminus \{b_1,b_2\})=0
\]
since it is equal to the inner product of $W$ with $A_+\delta_{S\setminus\{b_1,b_2\}}$ and $S\setminus\{b_1,b_2\}$ is in $\Sigma_{r-1}$.

In the last sum, each term $W(S\setminus\{ b_i\})$ occurs once for each other element of $S$, that is, it occurs 
$r$ times.  Since the vertex values that contribute to $S$ under $A_+$ are precisely those from vertices in $\Sigma_r$ obtained by 
deleting a single element of $S$, one has 
\[\sum_{b\notin S} \sum_{\{b_1,b_2\}\subset S\cup\{b\}, b_i\neq b} W((S\cup\{b\})\setminus \{b_1,b_2\})  
=- r (A_+W)(S)\, .
\]
Counting this twice gives the sum II in (\ref{mult:aminusaplus2}) and adding it to the sum I  gives a  total
\begin{multline*}(A_-A_+^2W)(S)= 2(N-r-1) (A_+W)(S)-2r(A_+W)(S)\\=[(N-2r)+(N-2(r+1))] (A_+W)(S)
\end{multline*}
as claimed. This proves the case $k=1$ of Thm.~\ref{thm:wradjacency}.

The proof of Thm.~\ref{thm:wradjacency} can be completed now 
by induction on $k$.  Suppose that $A_-A_+^k=m(r,k-1) A_+^{k-1}$
has been established on $\mathcal{W}_r$. We have, for $W\in \mathcal{W}_r$
\begin{multline*}
A_-A_+^{k+1}W =(A_-A_+ A_+^{k} -A_+A_- A_+^{k}+A_+A_- A_+^{k})W\\
=CA_+^{k}W+A_+A_- A_+^{k})W
=(N-2(k+r))A_+^k W+m(r,k-1)A_+^kW\\=m(r,k) A_+^k W
\end{multline*}
where we used Lem.~\ref{lem:commutator}, the induction hypothesis, and the fact that $m(r,k-1)+(N-2(r+k))=m(r,k)$.
This completes the proof of Thm.~\ref{thm:wradjacency}.\end{proof}

\begin{proof}[Proof of Lem.~\ref{lem:commutator}]
The argument is similar to that used in the case $k=0$ of the proof of 
Thm.~\ref{thm:wradjacency}.
 For $R\in \Sigma_r$, the value $(A_-A_+ V)(R)$ is the sum
of the values $V(P)$, $P\in\Sigma_r$ such that there is a two-edge path from $P$ to $R$ that passes
through a vertex in $\Sigma_{r+1}$.  This sum can be expressed as 
\[(A_-A_+ V)(R)=(N-r) \, V(R)+\sum_{\beta\notin R,\, b\in R} V(R\cup\{\beta\}\setminus \{b\})\]
where the first term counts the number of elements not in $R$, corresponding to paths from $R$ to a
vertex in $\Sigma_{r+1}$ and back to $R$, and the second counts all other paths, which necessarily come
from one-bit substitutions of bits in $R$ by bits not in $R$.

On the other hand,  the value $(A_+A_- V)(R)$ is the sum
of the values $V(P)$, $P\in\Sigma_r$ such that there is a two-edge path from $P$ to $R$ that passes
through a vertex in $\Sigma_{r-1}$. This sum can be expressed as
\[(A_+A_- V)(R)=r \, V(R)+\sum_{\beta\notin R,\, b\in R} V((R\setminus \{b\})\cup\{\beta\})\]
where the first term counts the number of paths from $R$ to $\Sigma_{r-1}$ and back to $R$ (which is the number of
ways to delete an element of $R$) and the second counts the number of paths that first delete an element of
$R$ then add an element not in $R$.

Subtracting the two terms gives
\[(CV)(R)=(A_-A_+ V)(R)-(A_+A_- V)(R)=(N-2r)V(R)\,.
\]
This proves the lemma.
\end{proof}

As a consequence of Theorem \ref{thm:wradjacency}, the action of any sequence of powers of $A_+$ and $A_-$ can be computed.  
For example, if $W\in \mathcal{W}_r$ then
\[A_-^2 A_+^{k+1}W=m(r,k) A_- A_+^{k} W=m(r,k) m(r,k-1) A_+^{k-1} W \, .\]
In general, if $B=A_-^{\tau_p} A_+^{\sigma_{p}}\cdots A_{-}^{\tau_1} A_+^{\sigma_1}$ then for
$W\in\mathcal{W}_r$, $BW$ is defined on $\Sigma_q$ where $q=r+(\sigma_1+\cdots+\sigma_p)-(\tau_1+\cdots+\tau_p)$
and is a multiple of $A_+^{q-r}W$ if $q\geq r$. 
However, if in any of the partial compositions of powers of $A_+$ and $A_-$ defining $B$, the number of applications of $A_+$ exceed those of $A_-$ 
by more than $N-r$ or the number of applications of $A_-$ exceeds those of $A_+$, then $BW=0$.

Higher order commutators of $A_-$ and $A_+$ simplify as follows.
Again, it suffices to consider a vector $V$ supported in $\Sigma_r$ for some $r$. In this case, $CV=(N-2r)V$
and $A_-CV=(N-2r)A_-V$ whereas $A_-V$ is supported in $\Sigma_{r-1}$ so $CA_-V=(N-2(r-1)) A_-V$ by 
Lem.~\ref{lem:commutator}.  Consequently,
\begin{multline*}[A_-,C]=(A_-C-CA_-)V=(N-2r)A_-V -(N-2(r-1)) A_-V
=-2A_-V\, \, {\rm and} \\
[A_+,C]=(A_+C-CA_+)V=(N-2r)A_+V -(N-2(r+1)) A_+V
=2A_+V\, .
\end{multline*}
It follows from Lem.~\ref{lem:commutator}  
that $[A_+,[A_-,C]]=-2[A_+,A_-]=2C$ and $[A_-,[A_+,C]]=2[A_-,A_+]=2C$ as well.
In particular, any higher order commutators involving $A_-$ and $A_+$ reduce to mutliples of $A_-$, $A_+$ and $C$
themselves, or vanish as in the case $[A_-,[A_-,C]]=-2[A_-,A_-]=0$.

\subsection{Projection onto $\mathcal{W}_r$}

Theorem \ref{thm:wradjacency} can be used to compute the projection of an arbitrary element of $\ell^2(\Sigma_r)$, that is,
a vertex function $V_r$ on $\mathcal{B}_N$ supported in $\Sigma_r$, onto the space $\mathcal{W}_r$ of vectors orthogonal
to $A_+(\ell^2(\Sigma_{r-1}))$.  This is done iteratively: first one computes the orthogonal complement of vectors of the 
form $A_+^rV_0$ where $V_0$ is supported in $\Sigma_0$ (this is the same as subtracting the average of $V_r$ since
$A_+^rV_0$ is constant on $\Sigma_r$). Call this orthogonal projection $V_r^{(0)}$. Next one projects $V_r^{(0)}$ onto the 
orthogonal complement of $A_+^{r-1}(\mathcal{W}_1)$ and continues until one has subtracted orthogonal projections
onto $\mathcal{W}_\ell$ for each $\ell<r$.  The result is a decomposition
\[V_r=\sum_{k=0}^r A_+^{r-k} W_k,\quad W_k\in\mathcal{W}_k\, .
\]
Working backwards, one finds that 
\[W_0= A_-^{r} V_r/(m(0,r-1)m(0,r-2)\cdots m(0,0))
\]
where $m(r,k)$ is as in (\ref{eq:mrk}).  
Similarly,
\[W_1=A_-^{r-1} (V_r-A_+^rW_0)/(m(1,r-2)m(1,r-3)\cdots m(1,0))
\]
etc.   The projections must be computed iteratively starting with the innermost projections in order to ensure
that components extracted at each step are orthogonal to components from smaller spheres.

\smallskip\noindent{\bf Observation.} The columns of the  matrix  $P_{\mathcal{W}_r}$ of the projection onto $\mathcal{W}_r$ in the standard basis appear to form a Parseval frame for $
\mathcal{W}_r$. The norm of each column is $1-\binom{N}{r-1}/\binom{N}{r}$ and the inner product of  columns $R$ and $S$ (equal to $P_{\mathcal{W}_r}(R,S)$ since $P_{\mathcal{W}_r}^2=P_{\mathcal{W}_r}$) depends only on $|R\cap S|$.

\begin{figure}[ht]
\centering 
\includegraphics[width=3.5in,height=2.2in]{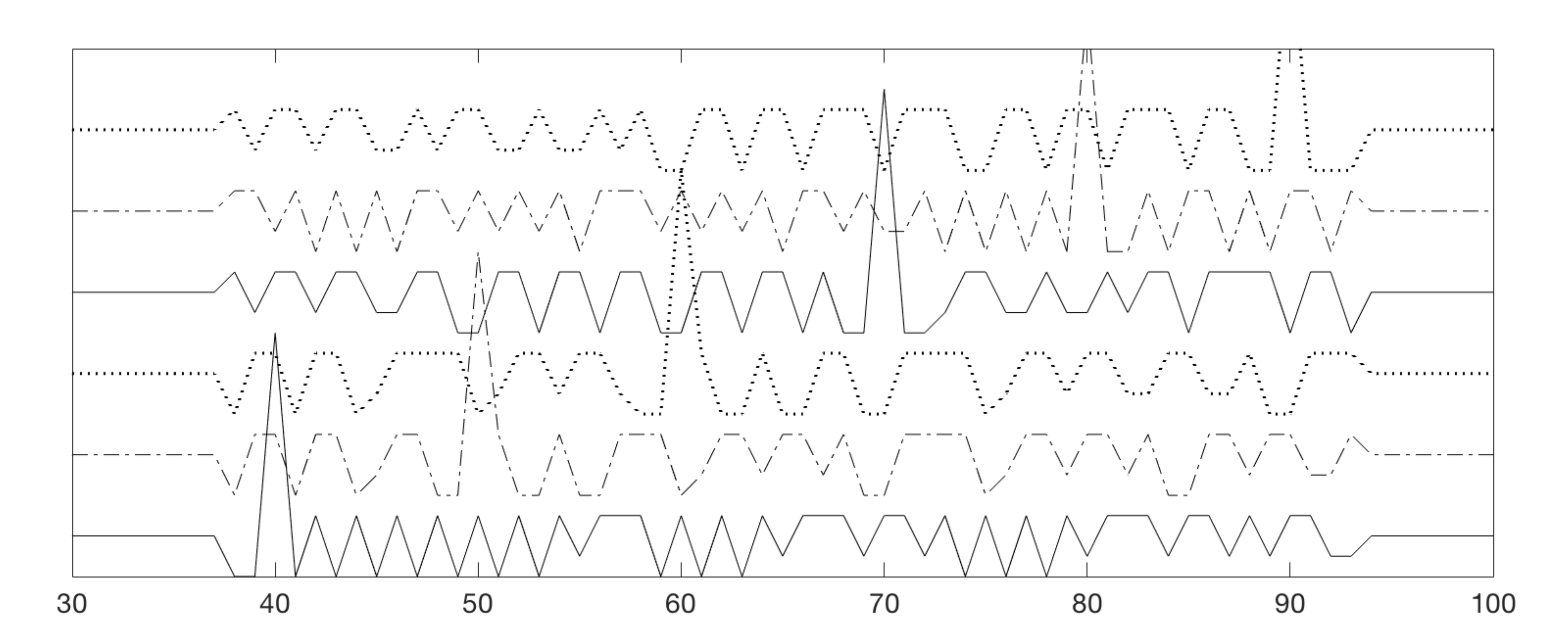}
\caption{Sample columns of the matrix of the projection onto $\mathcal{W}_r$, $N=8$, $r=3$.}
\label{fig:wrproj_vectors_100618}
\end{figure}

\section{Eigenspaces of Boolean difference operators}

Define $\mathcal{V}_r$ to consist of those vectors $V$ such that 
\[V=\sum_{k=0}^{N-r} c_k A_+^k W;\quad W\in \mathcal{W}_r\, .
\]
That is, $V$ is a sum of multiples of nonnegative powers of $A_+$ applied to a fixed vector $W\in \mathcal{W}_r$ and so
$V$ is a fixed multiple of $A_+^{\rho-r} W$ on $\Sigma_\rho$, $\rho\geq r$.
Theorem \ref{thm:hbdo_matrix} states that on $\mathcal{V}_r$,  conjugation of the Boolean difference operator BDO in (\ref{eq:bdo}) by the normalized Hadamard matrix $\bar{H}$, is
equivalent to multiplying the spherical coefficients $[c_0,\dots, c_{N-r}]$ of $V\in\mathcal{V}_r$ by a certain
tri-diagonal matrix $M$ of size $N-r+1$.

Define {\rm HBDO} to be the conjugation of BDO by $\bar{H}$, that is ${\rm HBDO}=2^{-N}H{\rm BDO} H$. Since $D=\bar{H}T\bar{H}$, 
HBDO can be written as $T (\alpha-L)T+\beta L$ where, as before, $T$ is the diagonal matrix with entries %
$T_{RR}=\sqrt{2r}$ when $|R|=r$.
Consider now the action of {\rm HBDO} on a vector $V=\sum_{k=0}^N V_k$ in $\mathcal{V}_r$ where $V_k$ is the restriction of $V$ to $
\Sigma_k$.  Since $L=NI-A$, it maps $V_k$ into a multiple by $N$ on $\Sigma_k$ and parts $A_\pm V_k$ supported in $\Sigma_{k\pm 1}$. One has
\[({\rm HBDO}V_k)_k= (T (\alpha-N I)T+\beta N I) V_k =[2k(\alpha-N)+\beta N] V_k\, .
\] 
The contribution of ${\rm HBDO}V_k$ to other spheres comes from the part
\[T AT-\beta A=(T A_+T-\beta A_+)+T A_-T-\beta A_-\]
which expresses the parts that map $V_k$ into $\Sigma_{k+1}$ and $\Sigma_{k-1}$ respectively.
One has
\begin{eqnarray*}(T A_+T-\beta A_+) V_k&=&(2\sqrt{k(k+1)}-\beta)\, A_+ V_k\quad{\rm and}\quad \\
 (T A_-T-\beta A_-) V_k&=&(2\sqrt{k(k-1)}-\beta)\, A_- V_k\, .
 \end{eqnarray*}
Since $V_k=c_{k}\, A_+^{k-r}V_r$, $V_r\in\mathcal{W}_r$,  
\begin{multline*}(T A_+T-\beta A_+) V_k=c_k ((2\sqrt{k(k+1)}-\beta)\, A_+ ) A_+^{k-r}V_r\\
=c_k (2\sqrt{k(k+1)}-\beta)\, A_+^{k+1-r}V_r =(c_k/c_{k+1}) (2\sqrt{k(k+1)}-\beta) V_{k+1}\end{multline*}
if $k+1\leq N$. If $k=N$ then this term vanishes.
Similarly,
\begin{multline*} (T A_-T-\beta A_-) V_k=c_k (2\sqrt{k(k-1)}-\beta) A_-  A_+^{k-r}V_r \\
= c_k (2\sqrt{k(k-1)}-\beta) m(r,k-r-1) A_+^{k-1-r}V_r \\
= (c_k/c_{k-1}) (2\sqrt{k(k-1)}-\beta) m(r,k-r-1) V_{k-1}\end{multline*}
where, as before, $m(r,k-r-1)$ is defined by (\ref{eq:mrk}) 
when $k-1\geq r$. If $k=r$ then this term vanishes. 

Putting these pieces together, one considers the coefficient of $V_k$ in the image under {\rm HBDO}.  
The calculations above indicate that HBDO is represented on $\mathcal{V}_r$ by  a matrix $M_r(k,\ell)$ of size $N-r+1$ 
acting on the vectors of coefficients $[c_0,\dots, c_{N-r}]^T$ of powers $A_+^k$ of $W\in \mathcal{W}_r$.
One then considers the mapping as a \emph{discrete} mapping on the coefficients of the vectors
$A_+^{k-r}V_r$. 
It follows from the calculations above that $M_r(k,\ell)$ can be expressed as the lower-right minor ($M_r(k,\ell)=M(k+r,\ell+r)$) of the size $N+1$ tri-diagonal matrix $M$ with entries:
\begin{equation}\label{eq:mhbdo} 
M(k, \ell)=\begin{cases} (2\sqrt{\ell(\ell-1)}-\beta) m(r,\ell-1-r);\,\, k=\ell-1\geq r\\
2\ell (\alpha-N)+\beta N; k=\ell\geq r\\
2\sqrt{\ell (\ell+1)}-\beta; k=\ell+1;  r\leq \ell<N\\
0,\,\,{\rm else}\, .
\end{cases}
\end{equation}
We have proved the following.

\begin{theorem}\label{thm:hbdo_matrix} If $V\in\mathcal{V}_r$, $ V=\sum_{k=0}^{N-r} c_k A_+^k W$, then 
${\rm HBDO} V=\sum_{k=0}^{N-r} d_k A_+^k W$ where $\mathbf{d}=M_r\mathbf{c}$ where $\mathbf{c}=[c_0,\dots, c_{N-r}]^T$.
\end{theorem}

If $\mathbf{c}=[c_0,\dots, c_{N-r}]^T$ is an eigenvector of $M_r$ then $\sum c_k A_+^{k-r} V_r$ ($V_r\in\mathcal{W}_r$)
is an eigenvector of {\rm HBDO} having the same eigenvalue. 

\begin{corollary} Any eigenvector of ${\rm HBDO}$ has the form
$V=\sum_{k=0}^{N-r} c_k A_+^k W$ where $\mathbf{c}$ is an eigenvector of the matrix $M_r$.   
  Since $\mathcal{W}_r$ has dimension $\binom{N}{r}-\binom{N}{r-1}$, for $\alpha,\beta$ such that $M_r$ is nondegenerate, 
  {\rm HBDO} has $N-r+1$ eigenspaces of dimension $\binom{N}{r}-\binom{N}{r-1}$.
\end{corollary}

The theorem and corollary apply to HBDO without regard to the parameters $\alpha,\beta$ which can be connected
to bandlimiting properties.  In \cite{Hogan20172} it was shown that when $\alpha=\beta=2\sqrt{K(K-1)}$,
the bandlimiting operator $P_K$ commutes with BDO.  Since conjugation of $P_K$ by $\bar{H}$ is 
$Q_K$, this means that HBDO commutes with $Q_K$ when $\alpha=\beta=2\sqrt{K(K-1)}$ in HBDO.
In this case, one can define the reduced space $\mathcal{V}_{r,K}$ of vectors of the form $\sum_{k=0}^{K-r} c_k A_+^{k}W$, 
$W\in\mathcal{W}_r$ and the reduced matrix $M_{r,K}$, the restriction of $M_r$ to its principal minor of size $K+1-r$. 
One has the following.

\begin{proposition} For each $N$ and $r\leq K<N$ with $\alpha=\beta=\sqrt{K(K-1)}$, the
tridiagonal matrix $M_{r,K}$  of size $(K+1-r)$ 
has eigenvalues equal to those of the conjugation of HBDO by $Q_K$ applied to  
the space $\mathcal{V}_{r,K}$.
\end{proposition}

\begin{tiny}
\begin{table}[htp]
\caption{Matrix $M=M^{\rm HBDO}_{8,3,2}$ of HBDO on $\mathcal{V}_r$ with $(N,K,r)=(8,3,2)$
}
\begin{center}
\begin{tabular}{| |c|c|c|c|c|c | c|| }
\hline 
 51.1384  & -8.1169 &   0& 0 &  0  &  0  &       0 \\
   -2.0292  &  48.9948 &   0 & 0&  0 &  0 &      0\\
 0  &  0   &46.8513     &12.0964   &0    &0       &  0\\
0   &0   & 2.0161  &44.7077   &16.1050   &0      &  0\\
0  & 0   &0  & 4.0262   &42.5641   &0        &0\\
0  &0   &0   &0    & 6.0333   &40.4205       &   -48.2306\\
         0        & 0        & 0        & 0      &   0         &8.0384         &38.2769 \\ \hline \hline
\end{tabular}
\end{center}
\label{tab:mhbdo832}
\end{table}%
\end{tiny}

\section{Eigenspaces of $PQP$}
As before we assume that $Q=Q_K$ is multiplication by the characteristic function of the closed ball of radius $K$ centered
at the origin. Then $\bar{H}Q\bar{H}=P$ and conjugation of $PQP$ by $\bar{H}$ leaves the operator $QPQ$.  We study the eigenspaces
of the latter.

\begin{lemma} \label{lem:vr_invariance} $A_+$ and $A_-$ map $\mathcal{V}_r$ to itself.
\end{lemma}

For $W\in\mathcal{W}_r$, $A_+$ maps $\sum_{k=0}^{N-r} c_k A_+^{k} W$ to 
$\sum_{k=0}^{N-r-1} c_k A_+^{k+1} W = \sum_{k=1}^{N-r} c_{k-1} A_+^{k} W$ which is an element of $\mathcal{V}_r$ whose coefficient 
of $W=A_+^0 W$ is zero  while $A_-$ maps 
$\sum_{k=0}^{N-r} c_k A_+^{k} W$ to $\sum_{k=1}^{N-r} c_k m(r,k-1) A_+^{k-1} W = \sum_{k=0}^{N-r-1} c_{k+1} m(r,k) A_+^{k} W$
which is an element of $\mathcal{V}_r$ whose coefficient of $A_+^{N-r}W$ is zero. 
The lemma implies that $A=A_++A_-$ preserves $\mathcal{V}_r$ and so does any polynomial $p(A)$ in $A$.

\begin{proposition} \label{prop:P_poly}The spectrum-limiting operator $P=P_K$ can be expressed as a polynomial $p(A)$ of degree $N$.
\end{proposition}

\begin{proof} By Lem.~\ref{lem:hadamard_eigenvalue}, the Hadamard vectors are eigenvectors of $L$ and hence of $A$, with $AH_R=(N-2|R|)H_R$.
 One can then express $P$ in terms of $A$
simply by forming the Lagrange interpolating polynomial that maps the eigenvalues of $A$ to those of $P$ (which are one or zero).  
Recall that the Lagrange interpolating polynomial that maps $x_k$ to $y_k$, $k=0,\dots, N$ is $\sum p_k$ where
$p_k(x)=y_k\prod_{j=0,\,j\neq k}^N\frac{x-x_j}{x_k-x_j}$.   We choose $p$ so that $p(A) H_R=H_R$ if $|R|\leq K$ and $p(A)H_R=0$ if 
$|R| >K$.  Since $AH_R=(N-2|R|) H_R$ this means one should have $p(N-2r)=1$ if $0\leq r\leq K$ and $p(N-2r)=0$ if $r>K$.
Therefore, set
\begin{equation}\label{eq: pk_lagrange} p_k=\prod_{j=0,j\neq k}^N \frac{x-(N-2j)}{2(j-k)};\qquad p(x)=\sum_{k=0}^K p_k
\end{equation}
Then $P=p(A)$ as verified on the Hadamard basis. 
This proves the proposition.
\end{proof}

As a consequence of the proposition, the space $\mathcal{V}_r$ is invariant under $P$.  The action of $P$ on $\mathcal{V}_r$
can be quantified by a matrix $M_{P,r}$ of size $(N-r+1)$  as follows. 
Given any $W\in\mathcal{W}_r$, one can write
$P (A_+^k W)=\sum_{\ell=0}^{N-r} M_{P,r}(k,\ell) A_+^\ell W$  in which $M_{P,r}(k,\ell)$ is independent of $W$.  Then $M_{P,r}(k,\ell)$ is the 
$(k,\ell)$th entry of $M_{P,r}$ and if $V=\sum_{\ell=0}^{N-r} c_\ell A_+^\ell W$ ($W\in \mathcal{W}_r$) then $PV= \sum_{k=0}^{N-r} d_k A_+^k 
W$ where $d_k=\sum M_{P,r}(k,\ell) c_\ell$. 
In particular,  if $[c_0,\dots, c_{N-r}]^T$ is a $\lambda$-eigenvector of $M_{P,r}$ then
$V$ is a $\lambda$-eigenvector of $P$.

The eigenvectors of $PQP$ have the form $HV$ where $V$ is an eigenvector of $QPQ$ where $P=P_K$ and $Q=Q_K$ for the 
same $K$.  Eigenvectors of  $QPQ$ in $\mathcal{V}_r$ can be obtained from those of the matrix $M_{QPQ,r}$ which is just the 
$(K-r+1)$-principal minor of $M_{P,r}$.  This is because if $V\in \mathcal{V}_r$, $V=\sum_{k=0}^{N-r} c_k A_+^k W$, $W\in\mathcal{W}
_r$ then $QV=\sum_{k=0}^{K-r} c_k A_+^k W$: $Q$ kills those terms supported in $\Sigma_\rho$ where $\rho>K$.
Then $PQ V=\sum_{k=0}^{N-r} d_k A_+^k W$ where $d_k=\sum_{\ell=0}^{K-r} M_{P,r} (k,\ell) c_\ell$ and, finally, 
 $QPQ V=\sum_{k=0}^{K-r} d_k A_+^k W$ for the same $\mathbf{d}$. That, is, the coefficients of the powers of $A_+^k W$ that appear in $QPQ 
 V$ are obtained by applying the principal minor of $M_{P,r}$ to the coefficients of $A_+^\ell W$ for $\ell\leq K-k$. 
 These observations give us the following.
 \begin{corollary}\label{cor:qpq_evec}  If $[c_0,\dots c_{K-r}]^T$ is a $\lambda$-eigenvector of the principal minor of size $(K-r+1)$ of the 
 matrix $M_{P,r}$   described above then $V=\sum_{k=0}^{K-r} c_k A_+^k W$, $W\in \mathcal{W}_r$, is a $\lambda$-eigenvector of QPQ
 and $HV$ is a $\lambda$-eigenvector of $PQP$.
 \end{corollary}

Computing the eigenspaces of $PQP$ then boils down to computing the eigenvectors of the principal minors of the
matrices $M_{P,r}$.  

In what follows, we outline a formal procedure to compute the matrices $M_{P,r}$.  First, one expresses $P$ in terms of the Lagrange interpolation
polynomial $p(A)$ as defined in (\ref{eq: pk_lagrange}).  Next, one observes that the matrix of a power of $A$ operating
on $\mathcal{V}_r$ is the power of the matrix of $A$ operating on $\mathcal{V}_r$, that is, $M_{A^n,r}=(M_{A,r})^n$, $n=0,1,2,\dots$.   Consequently, one can replace each occurrence of $A-(N-2j)I_{2^N}$ in (\ref{eq: pk_lagrange}) by
$M_{A,r}-(N-2j)I_{N-r+1}$ to obtain the matrix $M_{P,r}$.

One can represent the actions of $A_+$ and $A_-$ on $\mathcal{V}_r$ as coefficient mappings in the proof of Lem.~\ref{lem:vr_invariance}. In particular,
$A_+$ acts simply as a \emph{shift} since 
\[A_+\sum_{k=0}^{N-r} c_k A_+^k W=\sum_{k=0}^{N-r} c_k A_+^{k+1} W=\sum_{k=1}^{N-r} c_{k-1} A_+^{k} W\]
where the upper index of the sum remains $N-r$ because if $W\in \mathcal{W}_r$ then $A_+^{N-r+1} W=0$.
One can thus represent $A_+$ through the matrix of size $N-r+1$ that maps $[c_0,\dots, c_{N-r}]$ to $[0,c_0,\dots c_{N-r-1}]$.  This is just the matrix $M_{A_+}$ on $\mathbb{C}^{N-r+1}$ having ones on the diagonal below the main diagonal and zeros elsewhere.

On the other hand, for $W\in\mathcal{W}_r$, 
\[A_-\sum_{k=0}^{N-r} c_k A_+^k W 
=
\sum_{k=0}^{N-r} c_k m(r,k-1) A_+^{k-1} W=\sum_{k=0}^{N-r-1} c_{k+1} m(r,k) A_+^{k} W\, .
\]
The matrix of $A_-$ thus maps $[c_0,\dots, c_{N-r}]$ to $[m(r,0) c_1,\dots, m(r,N-r-1) c_{N-r},0]$.
This is just the matrix $M_{A_-}$ having the value $M_{A_-}(k,k+1)=m(r,k)$ on the diagonal above the main
diagonal and zeros elsewhere. 

One can write the matrix of $A$ on $\mathbb{C}^{N-r+1}$ as $M_A=M_{A_+}+M_{A_-}$.  The actions of powers of
$A$ on $\mathcal{V}_r$ can then be expressed in terms of the corresponding powers of $M_A$ on $\mathbb{C}^{N+1-r}$ and the bandlimiting operator $P=P_K$ acting on $\mathcal{V}_r$ can be expressed as the matrix $M_{P_r}$
obtained substituting $M_A$ for $A$ in the interpolating polynomial $p(A)$ in Prop.~\ref{prop:P_poly}.
We summarize how to compute the action of the operator $QPQ$ on $\mathcal{V}_r$, assuming $M_{P_r}$ has nondegenerate eigenvalues, as follows. 
Once the eigenspaces of $QPQ$ are computed by identifying the eigenspaces of $QPQ$ lying in $\mathcal{V}_r$
for each $r$, a complete eigenspace decomposition of $PQP$ is obtained by conjugating with the normalized
Hadamard matrix.

\noindent{\bf Algorithm to compute the matrix and eigenspaces of $QPQ$ on $\mathcal{V}_r$}

\noindent{\tt Step 1}:  Write the coefficient matrix $M_{A_+}$ of $A_+$ as the matrix of size $(N-r+1)$ with ones on the diagonal below the main
diagonal and zeros elsewhere.  Write the matrix $M_{A_-}$ of $A_-$ as the matrix of size $(N-r+1)$ with entries $M_{A_-}(k,k+1)=m(r,k)$, $k=0,\dots N-r$, and zeros elsewhere. The matrix $M_A$ of $A$ is $M_{A_+}+M_{A_-}$.

\noindent{\tt Step 2}: Compute the coefficient matrices $M_{P,r}$ by substituting $M_{A,r}-(N-2j)I_{N-r+1}$ for
each occurrence of $A-(N-2j)I_{2^N}$ in (\ref{eq: pk_lagrange}).

\noindent{\tt Step 3}:  Estimate the eigenvalues and eigenvectors of $M_{P,r}$ using standard numerical methods.
Each eigenspace has dimension $\binom{N}{r}-\binom{N}{r-1}$. A basis for each corresponding eigenspace  can be identified starting with a basis for $\mathcal{W}_r$ and forming the expansions $\sum_{k=0}^{K-r} c_k A_+^k W$
where $[c_0,\dots, c_{K-r}]^T$ is one of the  eigenvectors of $M_{P,r}$ and $W$ is one of the basis vectors for $\mathcal{W}_r$.

\smallskip\noindent{\bf Observation.} The adjacency matrix $A$ has norm $N$ and the norm of $M_{P,r}$ scales like $N^{\alpha N}$ for some $\alpha>1/2$. 
$M_{P,r}$ is poorly conditioned: {\tt matlab} returns {\tt Inf} for the condition number of $M_{P,1}$ when $N\geq 10$ and $K\leq 4$.

\begin{tiny}
\begin{table}[htp]
\caption{Matrix $M=M_{8,3,1}$ of $P$ on $\mathcal{V}_r$ with $(N,K,r)=(8,3,1)$}
\begin{center}
\begin{tabular}{| |c|c|c|c|c|c | c|c|| }
\hline 
    0.3437   & 0.1562 &   0.0156&   -0.0052 &  -0.0013 &   0.0003 &   0.0002 &        1\\
    0.9375  &  0.5000 &   0.0937 &   0.0000 &  -0.0026 &   0.0000 &   0.0003 &        1\\
    0.9375  &  0.9375 &   0.4062 &   0.0938 &   0.0078 &  -0.0026 &  -0.0013 &        1\\
   -3.7500  & 0.0000 &   1.1250 &   0.5000 &   0.0938 &   0.0000 &  -0.0052 &        1\\
  -11.2500  &  -3.7500 &   1.1250 &   1.1250 &   0.4062 &   0.0938 &   0.0156 &        1\\
   22.5000  & 0.0000  & -3.7500  &  0.0000 &   0.9375 &   0.5000  &  0.1562  &       1\\
  112.5000 &  22.5000 & -11.2500 &  -3.7500 &   0.9375 &   0.9375 &   0.3438 &        1\\
         0    &     0 &        0 &        0&         0 &        0  &       0  &       1 \\ \hline\hline
\end{tabular}
\end{center}
\label{tab:mhbdo831}
\end{table}%
\end{tiny}

\begin{tiny}
\begin{table}[htp]
\caption{Matrix $M=M_{8,3,2}$ of $P$ on $\mathcal{V}_r$ with $(N,K,r)=(8,3,2)$}
\begin{center}
\begin{tabular}{| |c|c|c|c|c|c | c|| }
\hline 
 0.3125  &  0.1875 &   0.0312 &  -0.0104 &  -0.0078  &  0.0029  &       0.0005 \\
    0.7500 &   0.5000 &   0.1250 &  -0.0000 &  -0.0104 &   0.0029 &        0.0006\\
    0.7500  &  0.7500    &0.3750    &0.1250    &0.0312    &0.0015       &  0\\
   -1.5000   &-0.0000    &0.7500   & 0.5000    &0.1875   & 0.0032       &  -0.0030\\
   -4.5000  & -1.5000    &0.7500   & 0.7500    &0.3125   & 0.0938         &0.0121\\
  0  & 0    &0  &0   &0   &0.4964       &  0.1754\\
         0        & 0        & 0        & 0      &   0         &1.0357         &0.4503 \\ \hline \hline

\end{tabular}
\end{center}
\label{tab:mp832}
\end{table}%
\end{tiny}

\begin{tiny}
\begin{table}[htp]
\caption{Matrix $M=M_{8,3,3}$ of $P$ on $\mathcal{V}_r$ with $(N,K,r)=(8,3,3)$}
\begin{center}
\begin{tabular}{| |c|c|c|c|c | c|| }
\hline 
  0.2500 &   0.2500   & 0.1250 &   0.0530  &  0.0234   & 0.0012\\
  0.5000   & 0.5000    &0.2500   & 0.0559  &   0.0250  &  0.0007\\
   0.5000    &0.5000   & 0.2500   &  0.1255  &  0.0014 &  -0.0021\\
0 & 0  & 0 &  0.4929 &   0.0311  &  0.0112\\
0  & 0 &  0  &  1.1723  &  0.7162 &   0.1409\\
   0        & 0      &   0       &   1.0173    &2.6598 &   0.3792\\ \hline\hline
\end{tabular}
\end{center}
\label{tab:mp833}
\end{table}%
\end{tiny}

\begin{figure}[ht]
\centering 
\includegraphics[width=3.5in,height=2.2in]{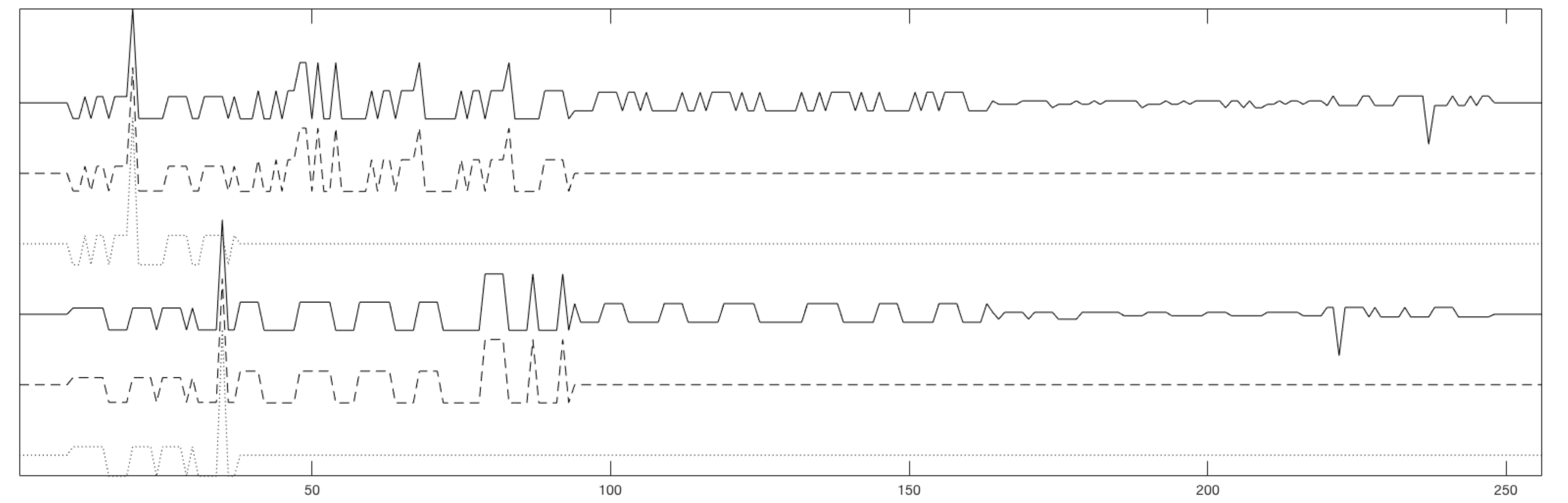}
\caption{Eigenvectors of $PQP$, $N=8$, $K=3$, $r=2$.  The dotted curves are plots of two different elements $W$ of $\mathcal{W}_r$ obtained by projecting a delta in $\Sigma_r$ onto $\mathcal{W}_r$ .  
The dashed curves are horizontal shifts of eigenvectors $V$ of $QPQ$
of the form $\sum c_k A_+^k W$  where $c_k$ form the principal eigenvector of $M_{r,K}$.  The solid curves are shifted multiples of $HV$ where $H$ is the Hadamard matrix. They are the corresponding eigenvectors of $PQP$.}
\label{fig:pqpevecsN8K3R2_proj_20_35}
\end{figure}

\bigskip
\noindent\textsc{Acknowledgement.} The second author would like to thank Beth Pollack for making this work possible and Anna Gilbert for suggesting the problem.
 \vfill\eject
 
 \thispagestyle{empty}

\bibliographystyle{plain}
\bibliography{boolean_eigenspace_refs.bib}

\end{document}